\documentclass[12pt]{amsart}
\usepackage{amsmath,amsthm,amsfonts,amssymb,mathrsfs}
\date{\today}

\usepackage{color}

\newtheorem{theorem}{Theorem}[section]
\newtheorem{proposition}[theorem]{Proposition}
\newtheorem{corollary}[theorem]{Corollary}
\newtheorem{lemma}[theorem]{Lemma}

\theoremstyle{definition}
\newtheorem{example}[theorem]{Example}
\newtheorem{remark}[theorem]{Remark}
\newtheorem{definition}[theorem]{Definition}
\newtheorem{question}[theorem]{Question}

\usepackage{hyperref}

\begin{document}

\title[Topological semigroups of matrix units and countably
compact Brandt $\lambda^0$-extensions~]{Topological semigroups of
matrix units and countably compact Brandt $\lambda^0$-extensions
of topological semigroups}

\author{Oleg~Gutik}
\address{Department of Mechanics and Mathematics, Ivan Franko Lviv National
University, Universytetska 1, Lviv, 79000, Ukraine}
\email{o\underline{\hskip5pt}\,gutik@franko.lviv.ua,
ovgutik@yahoo.com}

\author{Kateryna~Pavlyk}
\address{Pidstryhach Institute for Applied
Problems of Mechanics and Mathematics of National Academy of
Sciences, Naukova 3b, Lviv, 79060, Ukraine \, and \, Department of
Mechanics and Mathematics, Ivan Franko Lviv National University,
Universytetska 1, Lviv, 79000, Ukraine}
\email{kpavlyk@yahoo.co.uk}

\author{Andriy~Reiter}
\address{Department of Mechanics and Mathematics, Ivan Franko Lviv National
University, Universytetska 1, Lviv, 79000, Ukraine}
\email{reiter\underline{\hskip5pt}\,andriy@yahoo.com, reiter@i.ua}



\begin{abstract}
We show that a topological semigroup of finite partial bijections
$\mathscr{I}_\lambda^n$ of an infinite set with a compact
subsemigroup of idempotents is absolutely $H$-closed and any
countably compact topological semigroup does not contain
$\mathscr{I}_\lambda^n$ as a subsemigroup. We give sufficient
conditions onto a topological semigroup $\mathscr{I}_\lambda^1$ to
be non-$H$-closed. Also we describe the structure of countably
compact Brandt $\lambda^0$-extensions of topological monoids and
study the category of countably compact Brandt
$\lambda^0$-extensions of topological monoids with zero.
\end{abstract}

\maketitle


\section{Introduction and preliminaries}

In this paper all spaces are Hausdorff. Further we follow the
terminology of \cite{CHK, CP, Engelking1989}. By $\omega$ we
denote the first infinite cardinal. If $Y$ is a subspace of a
topological space $X$ and $A\subseteq Y$, then by
$\operatorname{cl}_Y(A)$ we denote the topological closure of $A$
in $Y$.

An algebraic semigroup $S$ is called {\it inverse} if for any
element $x$ in $S$ there exists the unique $x^{-1}\in S$ such that
$xx^{-1}x=x$ and $x^{-1}xx^{-1}=x^{-1}$. The element $x^{-1}$ is
called {\it inverse to} $x\in S$. If $S$ is an inverse semigroup,
then the function $\operatorname{inv}\colon S\to S$ which assigns
to every element $x$ of $S$ an inverse element $x^{-1}$ is called
{\it inversion}.

If $S$ is a semigroup, then by $E(S)$ we denote the band (the
subset of idempotents) of $S$, and by $S^1$ [$S^0$] we denote the
semigroup $S$ with the adjoined unit [zero] (see \cite{CP}). Also
if a semigroup $S$ has zero $0_S$, then for any $A\subseteq S$ we
denote $A^*=A\setminus\{ 0_S\}$. For an inverse semigroup $S$ we
define the maps $\varphi\colon S\rightarrow E(S)$ and $\psi\colon
S\rightarrow E(S)$ by the formulae $\varphi(x)=x\cdot x^{-1}$ and
$\psi(x)=x^{-1}\cdot x$.

If $E$ is a semilattice, then the semilattice operation on $E$
determines the partial order $\leqslant$ on $E$: $$e\leqslant
f\quad\text{if and only if}\quad ef=fe=e.$$ This order is called
{\em natural}. An element $e$ of a partially ordered set $X$ is
called {\em minimal} if $f\leqslant e$  implies $f=e$ for $f\in
X$. An idempotent $e$ of a semigroup $S$ without zero (with zero)
is called \emph{primitive} if $e$ is a minimal element in $E(S)$
(in $(E(S))^*$).

A {\it topological} ({\it inverse}) {\it semigroup} is a
topological space together with a continuous multiplication (and
an inversion, respectively).

A topological space $X$ is called \emph{countably compact} if any
countable open cover of $X$ contains a finite
subcover~\cite{Engelking1989}. A topological space $X$ is called
\emph{pseudocompact} (\emph{discretely pseudocompact}) if every
locally finite (discrete) family of non-open subsets of $X$ is
finite~\cite{Engelking1989}. A Tychonoff topological space $X$ is
\emph{pseudocompact} if and only if each continuous real-valued
function on $X$ is bounded~(see
\cite[Theorem~3.10.22]{Engelking1989}). Obviously that every
countably compact space is pseudocompact and every pseudocompact
space is discretely pseudocompact. Also we observe that every
quasi-regular discretely pseudocompact space is pseudocompact. We
recall that the Stone-\v{C}ech compactification of a Tychonoff
space $X$ is a compact Hausdorff space $\beta{X}$ containing $X$
as a dense subspace so that each continuous map $f\colon
X\rightarrow Y$ to a compact Hausdorff space $Y$ extends to a
continuous map $\overline{f}\colon \beta{X}\rightarrow
Y$~\cite{Engelking1989}.

Let $S$ be a semigroup with zero and $\lambda$ be cardinal
$\geqslant 1$. On the set $B_{\lambda}(S)=\lambda\times S\times
\lambda\cup\{ 0\}$ we define the semigroup operation as follows
\begin{equation*}
 (\alpha,a,\beta)\cdot(\gamma, b, \delta)=
  \begin{cases}
    (\alpha, ab, \delta), & \text{ if } \beta=\gamma; \\
    0, & \text{ if } \beta\ne \gamma,
  \end{cases}
\end{equation*}
and $(\alpha, a, \beta)\cdot 0=0\cdot(\alpha, a, \beta)=0\cdot
0=0,$ for all $\alpha, \beta, \gamma, \delta\in\lambda$ and $a,
b\in S$. If $S=S^1$ is a semigroup with unit then the semigroup
$B_\lambda(S)$ is called the {\it Brandt $\lambda$-extension of
the monoid} $S$~\cite{GutikPavlyk2001}. Obviously, ${\mathcal
J}=\{ 0\}\cup\{(\alpha, {\mathscr O}, \beta)\mid {\mathscr O}$ is
the zero of $S\}$ is an ideal of $B_\lambda(S)$. We put
$B^0_\lambda(S)=B_\lambda(S)/{\mathcal J}$ and we shall call
$B^0_\lambda(S)$ the {\it Brandt $\lambda^0$-extension of the
monoid $S$ with zero}~\cite{GutikPavlyk2006}. Further, if
$A\subseteq S$ then we shall denote $A_{\alpha,\beta}=\{(\alpha,
s, \beta)\mid s\in A \}$ if $A$ does not contain zero, and
$A_{\alpha,\beta}=\{(\alpha, s, \beta)\mid s\in A\setminus\{ 0\}
\}\cup \{ 0\}$ if $0\in A$, for $\alpha, \beta\in\lambda$. If
$\mathcal{I}$ is a trivial semigroup (i.e. $\mathcal{I}$ contains
only one element), then by ${\mathcal{I}}^0$ we denote the
semigroup $\mathcal{I}$ with the adjoined zero. Obviously, for any
$\lambda\geqslant 2$ the Brandt $\lambda^0$-extension of the
semigroup ${\mathcal{I}}^0$ is isomorphic to the semigroup of
$\lambda\times\lambda$-matrix units and any Brandt
$\lambda^0$-extension of a monoid with zero contains the semigroup
of $\lambda\times\lambda$-matrix units. Further by $B_\lambda$ we
shall denote the semigroup of $\lambda\times\lambda$-matrix units
and by $B^0_\lambda(1)$ the subsemigroup of
$\lambda\times\lambda$-matrix units of the Brandt
$\lambda^0$-extension of a monoid $S$ with zero.

Let $\mathscr{I}_\lambda$ denote the set of all partial one-to-one
transformations of a set $X$ of cardinality $\lambda$ together
with the following semigroup operation:
\begin{equation*}
    x(\alpha\beta)=(x\alpha)\beta \quad \mbox{if} \quad
    x\in\operatorname{dom}(\alpha\beta)=\{
    y\in\operatorname{dom}\alpha\mid
    y\alpha\in\operatorname{dom}\beta\}, \qquad \mbox{for} \quad
    \alpha,\beta\in\mathscr{I}_\lambda.
\end{equation*}
The semigroup $\mathscr{I}_\lambda$ is called the \emph{symmetric
inverse semigroup} over the set $X$~(see \cite{CP}). The symmetric
inverse semigroup was introduced by Wagner~\cite{Wagner1952} and
it plays a major role in the theory of semigroups.

We denote
 $
 \mathscr{I}_\lambda^n=\{
\alpha\in\mathscr{I}_\lambda\mid
\operatorname{rank}\alpha\leqslant n\},
 $
{} for $n=1,2,3,\ldots$. Obviously, $\mathscr{I}_\lambda^n$
($n=1,2,3,\ldots$) is an inverse semigroup,
$\mathscr{I}_\lambda^n$ is an ideal of $\mathscr{I}_\lambda$ for
each $n=1,2,3,\ldots$. Further, we shall call the semigroup
$\mathscr{I}_\lambda^n$ the \emph{symmetric inverse semigroup of
finite transformations of the rank $n$}. The elements of the
semigroup $\mathscr{I}_\lambda^n$ are called \emph{finite
one-to-one transformations} (\emph{partial bijections}) of the set
$X$. By
\begin{equation*}
\left(%
\begin{array}{cccc}
  x_1 & x_2 & \cdots & x_n \\
  y_1 & y_2 & \cdots & y_n \\
\end{array}%
\right)
\end{equation*}
we denote a partial one-to-one transformation which maps $x_1$
onto $y_1$, $x_2$ onto $y_2$, $\ldots$, and $x_n$ onto $y_n$, and
by $0$ the empty transformation. Obviously, in such case we have
$x_i\neq x_j$ and $y_i\neq y_j$ for $i\neq j$
($i,j=1,2,3,\ldots,n$). We observe that the the symmetric inverse
semigroup $\mathscr{I}_\lambda^1$ of finite transformations of the
rank $1$ is isomorphic to the semigroup of matrix units
$B_\lambda$.

A~semigroup $S$ is called {\it congruence-free} if it has only two
congruences: identical and universal~\cite{Schein1966}. Obviously,
a semigroup $S$ is congruence-free if and only if every
homomorphism $h$ of $S$ into an arbitrary semigroup $T$ is an
isomorphism "into" or is an annihilating homomorphism (i.~e. there
exists $c\in T$ such that $h(a)=c$ for all $a\in S$).

Let ${\mathscr S}$ be a class of topological semigroups.

\begin{definition}[\cite{GutikPavlyk2001, Stepp1969}]\label{def1}
A~semigroup $S\in{\mathscr S}$ is called {\it $H$-closed in
${\mathscr S}$}, if $S$ is a closed subsemigroup of any
topological semigroup $T\in{\mathscr S}$ which contains $S$ as a
subsemigroup. If ${\mathscr S}$ coincides with the class of all
topological semigroups, then the semigroup $S$ is called {\it
$H$-closed}.
\end{definition}

\begin{definition}[\cite{GutikPavlyk2003, Stepp1975}]\label{def3}
A~topological semigroup $S\in{\mathscr S}$ is called {\it
absolutely $H$-closed in the class ${\mathscr S}$} if any
continuous homomorphic image of $S$ into $T\in{\mathscr S}$ is
$H$-closed in ${\mathscr S}$. If ${\mathscr S}$ coincides with the
class of all topological semigroups, then the semigroup $S$ is
called {\it absolutely $H$-closed}.
\end{definition}

A~semigroup $S$ is called {\it algebraically closed in ${\mathscr
S}$} if $S$ with any semigroup topology $\tau$ such that $(S,
\tau)\in{\mathscr S}$ is $H$-closed in ${\mathscr
S}$~\cite{GutikPavlyk2001}. If ${\mathscr S}$ coincides with the
class of all topological semigroups, then the semigroup $S$ is
called {\it algebraically closed}. A~semigroup $S$ is called {\it
algebraically $h$-closed in ${\mathscr S}$} if $S$ with discrete
topology $\mathfrak{d}$ is absolutely $H$-closed in ${\mathscr S}$
and $(S, \mathfrak{d})\in{\mathscr S}$. If ${\mathscr S}$
coincides with the class of all topological semigroups, then the
semigroup $S$ is called {\it algebraically $h$-closed}.

Absolutely $H$-closed semigroups and algebraically $h$-closed
semigroups were introduced by Stepp in~\cite{Stepp1975}. There
they were called {\it absolutely maximal} and {\it algebraic
maximal}, respectively.

Many topologists have studied topological properties of
topological spaces of partial continuous maps $\mathscr{PC}(X,Y)$
from a topological space $X$ into a topological space $Y$ with
various topologies such as the Vietoris topology, generalized
compact-open topology, graph topology, $\tau$-topology, and others
(see \cite{Abd-AllahBrown1980, BoothBrown1978,
DiConcilioNaimpally1998, Filippov1996, Hola1998, Hola1999,
KumziShapiro1997, Kuratowski1955}). Since the set of all partial
continuous self-transformations $\mathscr{PC\!T}(X)$ of the space
$X$ with the operation composition is a semigroup,  many semigroup
theorists have considered the semigroup of continuous
transformations (see surveys \cite{Magill1975} and
\cite{GluskinScheinSnepermanYaroker1977}), or the semigroup of
partial homeomorphisms of an arbitrary topological space (see
\cite{Baird1977, Baird1977a, Baird1977b, Baird1979, Gluskin1977,
Mendes-GoncalvesSullivan2006, Sneperman1962, Yaroker1972}).
Be\u{\i}da~\cite{Beida1980}, Orlov~\cite{Orlov1974, Orlov1974a},
and Subbiah~\cite{Subbiah1987} have considered  semigroup and
inverse semigroup  topologies of semigroups of partial
homeomorphisms of some classes of topological spaces.  In this
context the results of our paper yield some notable results about
the topological behavior of the finite rank symmetric inverse
semigroups setting inside larger function space semigroups, or
larger semigroups in general.  For example, under reasonably
general conditions, the inverse semigroup of partial finite
bijections $\mathscr{I}_\lambda^n$ of rank $\leqslant n$ is a
closed subsemigroup of a topological semigroup which contains
$\mathscr{I}_\lambda^n$ as a subsemigroup.

Gutik and Pavlyk in \cite{GutikPavlyk2005} consider the partial
case of the semigroup $\mathscr{I}_\lambda^n$: an infinite
topological semigroup of $\lambda\times\lambda$-matrix units
$B_\lambda$. There they showed that an infinite topological
semigroup of $\lambda\times\lambda$-matrix units $B_\lambda$ does
not embed into a compact topological semigroup, $B_\lambda$ is
algebraically $h$-closed in the class of topological inverse
semigroups, described the Bohr compactification of $B_\lambda$ and
minimal semigroup and minimal semigroup inverse topologies on
$B_\lambda$.

Gutik, Lawson and Repov\v{s} in \cite{GutikLawsonRepovs2009}
introduced the notion of a semigroup with a tight ideal series and
investigated their closures in semitopological semigroups,
particularly inverse semigroups with continuous inversion. As a
corollary they show that the symmetric inverse semigroup of finite
transformations $\mathscr{I}_\lambda^n$ of infinite cardinal
$\lambda$ is algebraically closed in the class of
(semi)topological inverse semigroups with continuous inversion.
They also derive related results about the nonexistence of
(partial) compactifications of classes of considered semigroups.

In \cite{GutikReiter2009} Gutik and Reiter show that the
topological inverse semigroup $\mathscr{I}_\lambda^n$ is
algebraically $h$-closed in the class of topological inverse
semigroups. Also they prove that a topological semigroup $S$ with
countably compact square $S\times S$ does not contain the
semigroup $\mathscr{I}_\lambda^n$ for infinite cardinal $\lambda$
and show that the Bohr compactification of an infinite topological
semigroup $\mathscr{I}_\lambda^n$ is the trivial semigroup.

Gutik and Repov\v{s} in \cite{GutikRepovs} study algebraic
properties of Brandt $\lambda^0$-extensions of monoids with zero
and non-trivial homomorphisms between Brandt
$\lambda^0$-extensions of monoids with zero. Also they describe a
category whose objects are ingredients of the construction of
Brandt $\lambda^0$-extensions of monoids with zeros. There they
introduce finite, compact topological Brandt
$\lambda^0$-extensions of topological semigroups and countably
compact topological Brandt $\lambda^0$-extensions of topological
inverse semigroups in the class of topological inverse semigroups
and establish the structure of such extensions and non-trivial
continuous homomorphisms between such topological Brandt
$\lambda^0$-extensions of topological  monoids with zero. They
also describe a category whose objects are ingredients in the
constructions of finite (compact, countably compact) topological
Brandt $\lambda^0$-extensions of topological  monoids with zeros.

In this paper we show that a topological semigroup of finite
partial bijections $\mathscr{I}_\lambda^n$ of an infinite set with
a compact subsemigroup of idempotents is absolutely $H$-closed. We
prove that any countably compact topological semigroup and any
Tychonoff  topological semigroup with pseudocompact square do not
contain $\mathscr{I}_\lambda^n$ as a subsemigroup. Moreover every
continuous homomorphism from topological semigroup
$\mathscr{I}_\lambda^n$ into a countably compact topological
semigroup or Tychonoff topological semigroup with pseudocompact
square is annihilating. We give sufficient conditions onto a
topological semigroup $\mathscr{I}_\lambda^1$ to be non-$H$-closed
and show that the topological inverse semigroup
$\mathscr{I}_\lambda^1$ is absolutely $H$-closed if and only if
the band $E(\mathscr{I}_\lambda^1)$ is compact. Also we describe
the structure of countably compact Brandt $\lambda^0$-extensions
of topological monoids and establish the category of countably
compact Brandt $\lambda^0$-extensions of topological monoids with
zero.


\section{On the closure and embedding of the semigroup of matrix units}

\begin{lemma}\label{lemma1.1}
Let $E$ be a topological semilattice with zero $0$ such that every
non-zero idempotent of $E$ is primitive. Then every non-zero
element of $E$ is an isolated point in $E$. Moreover for the
infinite topological semilattice $E$ the following conditions are
equivalent:
\begin{enumerate}
    \item[$(i)$] $E$ is compact;
    \item[$(ii)$] $E$ is countably compact;
    \item[$(iii)$] $E$ is pseudocompact;
    \item[$(iv)$] $E$ is discretely pseudocompact;
    \item[$(v)$] $E$ is homeomorphic to the one-point Alexandroff
    compactification of the discrete space $X$ of cardinality
    $|E|$ with zero $0$ as the remainder.
\end{enumerate}
\end{lemma}

\begin{proof}
Let $x\in E^*$. Since $E$ is a Hausdorff topological semilattice,
for any open neighbourhood $U(x)\not\ni 0$ of the point $x$ there
exists an open neighbourhood $V(x)$ of $x$ such that $V(x)\cdot
V(x)\subseteq U(x)$. If $x$ is not an isolated point of $E$ then
$V(x)\cdot V(x)\ni 0$ which contradicts to the choice of the
neighbourhood $U(x)$. This implies the first assertion of the
lemma.

We observe that the implications $(i)\Rightarrow(ii)$,
$(ii)\Rightarrow(iii)$ and $(iii)\Rightarrow(iv)$ are trivial.

To show the implication $(iv)\Rightarrow(i)$ suppose that the
semilattice $E$ is discretely pseudocompact and $E$ satisfies the
assertion of lemma. Suppose to the contrary that $E$ is not
compact. Let $\mathscr{C}=\{ U_s\mid s\in\mathscr{J}\}$ be any
open cover of $E$ such that $\mathscr{C}$ does not contain a
finite subcover. Let $U_{s_0}\in\mathscr{C}$ such that $0\in
U_{s_0}$. We denote $A=E\setminus U_{s_0}$. Since the topological
semilattice $E$ is non-compact, the set $A$ is infinite. Put
$\mathscr{W}=\{A\}\bigcup\{\{x\}\mid x\in A\}$. Then $\mathscr{W}$
is an infinite discrete family of open non-empty subsets of $E$.
This contradicts to discrete compactness of $A$. The obtained
contradiction implies that $E$ is a compact semilattice.

Simple verifications show that if the semilattice $E$ is
homeomorphic to the one-point Alexandroff compactification of the
discrete space $X$ of cardinality $|E|$ with zero $0$ as the
remainder the semilattice operation is continuous. This implies
the implications $(v)\Rightarrow(i)$. Also the first assertion of
the lemma implies the implications $(i)\Rightarrow(v)$.
\end{proof}

\begin{lemma}\label{lemma1.2}
Let $T$ be a topological semigroup which contains the infinite
semigroup of matrix units $B_\lambda$ as a dense subsemigroup.
Then the following conditions hold:
\begin{itemize}
    \item[$(i)$] the zero $0$ of $B_\lambda$ is thw zero of $T$;
    \item[$(ii)$] if $T\setminus B_\lambda\neq\varnothing$ then
         $x^2=0$ for all $x\in T\setminus B_\lambda$; and
    \item[$(iii)$] $E(T)=E(B_\lambda)$.
\end{itemize}
\end{lemma}

\begin{proof}
$(i)$ The set $\{x\in T\mid x\cdot 0=0\cdot x=0\}$ is closed and
contains the dense subset $B_\lambda$, so coincides with $T$.

$(ii)$ Let $T\setminus B_\lambda\neq\varnothing$. Suppose to the
contrary that there exists $x\in T\setminus B_\lambda$ such that
$x^2=y\neq 0$. Then for any open neighbourhoods $U(y)$ and $U(0)$
of $y$ and $0$ such that $U(y)\cap U(0)=\varnothing$ there exists
an open neighbourhood $V(x)$ such that $V(x)\cdot V(x)\subseteq
U(y)$ and $V(x)\cap U(0)=\varnothing$. Since  the closure of
semilattice in a topological semigroup is subsemilattice (see
\cite[Corollary~19]{GutikPavlyk2003}) Theorem~9 of
\cite{Stepp1975} implies that the band $E(B_\lambda)$ is a closed
subsemigroup of $T$. Hence without loss of generality we can
assume that $V(x)\subseteq T\setminus E(B_\lambda)$. Since the
neighbourhood $V(x)$ of the point $x$ contains infinitely many
point from $B_\lambda\setminus E(B_\lambda)$, we have that $0\in
V(x)\cdot V(x)$. This contradicts to the assumption that $U(y)\cap
U(0)=\varnothing$. Therefore $x^2=0$.

$(iii)$ The statement follows from statement $(ii)$.
\end{proof}

\begin{theorem}\label{theorem1.3}
A topological semigroup of matrix units $B_\lambda$ with a compact
band $E(B_\lambda)$ is an $H$-closed topological semigroup.
\end{theorem}

\begin{proof}
Since the statement theorem is trivial in case when the set
$E(B_\lambda)$ is finite, we consider the case when the band
$E(B_\lambda)$ is infinite.

Suppose to the contrary that there exists a topological semigroup
$T$ which contains $B_\lambda$ as a non-closed subsemigroup.
Without loss of generality we can assume that $B_\lambda$ is a
dense subsemigroup of $T$ and $T\setminus
B_\lambda\neq\varnothing$. Let $x\in T\setminus B_\lambda$. By
Lemma~\ref{lemma1.2} we have that zero $0$ of the semigroup
$B_\lambda$ is zero in the topological semigroup $T$ and $x^2=0$.

Since $0\cdot x=x\cdot 0=0$ for any open neighbourhoods $U(x)$ and
$U(0)$ in $T$ of $x$ and $0$, respectively, such that $U(x)\cap
U(0)=\varnothing$, there exist open neighbourhoods $V(x)$ and
$V(0)$ in $T$ of $x$ and $0$, respectively, such that
\begin{equation*}
    V(0)\cdot V(x)\subseteq U(0), \quad V(x)\cdot V(0)\subseteq
    U(0), \quad V(x)\subseteq U(x) \quad \mbox{and} \quad
    V(0)\subseteq U(0).
\end{equation*}
Since by Lemma~\ref{lemma1.1} any non-zero idempotent of
$B_\lambda$ is an isolated point in $E(B_\lambda)$, compactness of
$E(B_\lambda)$ implies that the set $E(B_\lambda)\setminus V(0)$
is finite and $V(x)\cap E(B_\lambda)=\varnothing$. Since the
neighbourhood $V(x)$ contains infinitely many element of the
semigroup $B_\lambda$ and the set $E(B_\lambda)\setminus V(0)$ is
finite, there exists $(\alpha,\beta)\in V(x)$ such that either
$(\alpha,\alpha)\in V(0)$ or $(\beta,\beta)\in V(0)$. Therefore,
we have that at least one of the following conditions holds:
\begin{equation*}
    (V(x)\cdot V(0))\cap V(x)\neq\varnothing \qquad \mbox{and}
    \qquad
    (V(0)\cdot V(x))\cap V(x)\neq\varnothing.
\end{equation*}
This contradicts to the assumption that $U(x)\cap
U(0)=\varnothing$. The obtained contradiction implies the
statement of the theorem.
\end{proof}

Lemma~\ref{lemma1.1} and Theorem~\ref{theorem1.3} imply the
following:

\begin{corollary}\label{corollary1.3a}
A topological semigroup of matrix units $B_\lambda$ with a
discretely pseudocompact (pseudocompact, countably compact) band
$E(B_\lambda)$ is an $H$-closed topological semigroup.
\end{corollary}

By Theorem~1~\cite{Gluskin1955} the semigroup of matrix units
$B_\lambda$ is congruence free and hence any homomorphic image of
$B_\lambda$ is either the semigroup of matrix units or the trivial
semigroup. Since a continuous image of a compact space is a
compact space (see \cite[Thorem~3.1.10]{Engelking1989}),
Theorem~\ref{theorem1.3} implies the following:

\begin{theorem}\label{theorem1.4}
A topological semigroup of matrix units $B_\lambda$ with a compact
band $E(B_\lambda)$ is an absolutely $H$-closed topological
semigroup.
\end{theorem}

Lemma~\ref{lemma1.1} and Theorem~\ref{theorem1.4} imply the
following:

\begin{corollary}\label{corollary1.4a}
A topological semigroup of matrix units $B_\lambda$ with a
discretely pseudocompact (pseudocompact, countably compact) band
$E(B_\lambda)$ is an absolutely $H$-closed topological semigroup.
\end{corollary}

The following theorem shows that the converse statement to
Theorem~\ref{theorem1.3} is true when $B_\lambda$ is a topological
inverse semigroup.

\begin{theorem}\label{theorem1.5}
If $B_\lambda$ is an $H$-closed topological inverse semigroup,
then the band $E(B_\lambda)$ is compact.
\end{theorem}

\begin{proof}
Suppose the contrary: the band $E(B_\lambda)$ is a non-compact
subset in $(B_\lambda,\tau)$. By \cite[Lemma~4]{GutikPavlyk2005}
any non-zero element of the semigroup $B_\lambda$ is an isolated
point in $(B_\lambda,\tau)$ and hence there exists an open
neighbourhood $U(0)$ of zero $0$ in $(B_\lambda,\tau)$ such that
$A=E(B_\lambda)\setminus\big(E(B_\lambda)\cap U(0)\big)$ is an
infinite subset of $E(B_\lambda)$. Without loss of generality we
can assume that $A$ is countable. Next we enumerate the set $A$ by
positive integers: $A=\{(\alpha_1,\alpha_1), (\alpha_2,\alpha_2),
(\alpha_3,\alpha_3), \ldots\}$. Then $A$ is a closed subset of
$E(B_\lambda)$ and hence the continuity of the inversion in
$(B_\lambda,\tau)$ implies that
$I_A=\varphi^{-1}(E(B_\lambda)\setminus
A)\cup\psi^{-1}(E(B_\lambda)\setminus A)$ is an open subset of the
topological space $(B_\lambda,\tau)$.

Let $x\notin B_\lambda$. Put $S=B_\lambda\cup\{ x\}$. We extend
the semigroup operation from $B_\lambda$ onto $S$ as follows:
\begin{equation*}
    x\cdot x=y\cdot x=x\cdot y=0, \qquad \qquad  \mbox{~for all~}
    \quad y\in B_\lambda.
\end{equation*}
Simple verifications show that such defined operation is
associative.

Put $A_n=\{(\alpha_{2k-1},\alpha_{2k})\mid k=n, n+1, n+2,
\ldots\}$ for any positive integer $n$. We determine a topology
$\tau^*$ on $S$ as follows:
\begin{itemize}
    \item[$(i)$] for every $y\in B_\lambda$ the bases of
                 topologies $\tau$ and $\tau^*$ at $y$ coincide;
                 and
    \item[$(ii)$] $\mathscr{B}(x)=\{ U_n(x)=\{ x\}\cup A_n\mid
                 n=1,2,3,\ldots\}$ is the base of the topology
                 $\tau^*$ at $x$.
\end{itemize}
For any open neighbourhood $V(0)$ of zero $0$ such that
$V(0)\subseteq U(0)$ we have
\begin{equation*}
    V(0)\cdot U_n(0)=U_n(0)\cdot V(0)=U_n(0)\cdot U_n(0)=\{
    0\}\subseteq V(0).
\end{equation*}
We observe that the definition of the set $A_n$ implies that for
any non-zero element $(\alpha,\beta)$ of the semigroup $B_\lambda$
there exists the smallest positive integer $i_{(\alpha,\beta)}$
such that $(\alpha,\beta)\cdot
A_{i_{(\alpha,\beta)}}=A_{i_{(\alpha,\beta)}}\cdot(\alpha,\beta)=\{
0\}$. Then we have
\begin{equation*}
    (\alpha,\beta)\cdot U_{i_{(\alpha,\beta)}}(0)=
    U_{i_{(\alpha,\beta)}}(0)\cdot(\alpha,\beta)=\{ 0\}\subseteq
    V(0).
\end{equation*}
Therefore $(S,\tau^*)$ is a topological semigroup which contains
$(B_\lambda,\tau)$ as a dense subsemigroup. The obtained
contradiction implies that $E(B_\lambda)$ is a compact subset of
$(B_\lambda,\tau)$.
\end{proof}

Since the semigroup of matrix units is congruence-free,
Theorems~\ref{theorem1.3} and \ref{theorem1.5} imply the following
theorem:

\begin{theorem}\label{theorem1.5a}
A topological inverse semigroup $B_\lambda$ is $H$-closed if and
only if the band $E(B_\lambda)$ is compact.
\end{theorem}

Gutik and Pavlyk in \cite{GutikPavlyk2005} show that an infinite
semigroup of matrix units with the discrete topology is not
$H$-closed. The following proposition gives the sufficient
conditions on an infinite topological semigroup of matrix units to
be non-$H$-closed.

\begin{proposition}\label{proposition1.5a}
Let $\tau$ be a semigroup topology on an infinite semigroup of
matrix units $B_\lambda$. If there exists an open neighbourhood
$U(0)$ of zero in $(B_\lambda,\tau)$ such that
$\big(\varphi^{-1}(A)\cup\psi^{-1}(A)\big)\cap U(0)=\varnothing$
for some infinite subset $A$ of $E(B_\lambda)$, then
$(B_\lambda,\tau)$ is not an $H$-closed topological semigroup.
\end{proposition}

\begin{proof}
We observe that without loss of generality we can assume that the
set $A$ is countable.

Let $x\notin B_\lambda$. Put $S=B_\lambda\cup\{ x\}$. We extend
the semigroup operation from $B_\lambda$ onto $S$ as follows:
\begin{equation*}
    x\cdot x=x\cdot a=a\cdot x=0 \qquad \mbox{~for all~} \quad a\in
    B_\lambda.
\end{equation*}
Simple verifications show that such defined binary operation is
associative.

Further we enumerate the elements of the set $A$ by the positive
integers, i.~e. $A=\{(\alpha_i,\alpha_i)\mid i=1,2,3,\ldots\}$.
Let $A_n=\{(\alpha_{2k-1},\alpha_{2k})\mid k\geqslant n\}$ for
each positive integer $n$. A topology $\tau_0$ on $S$ is defined
as follows:
\begin{itemize}
    \item[$a)$] bases of topologies $\tau$ and $\tau_0$ coincide
                 at any point $a\in B_\lambda$;
    \item[$b)$] $\mathscr{B}(x)=\{U_n(0)=\{ x\}\cup A_n\mid n$
                 is an positive integer$\}$ is a base of the
                 topology $\tau_0$ at the point $x\in S$.
\end{itemize}
Such defined topology $\tau_0$ on $S$ implies that it is complete
to show that the semigroup operation on $(S,\tau_0)$ is continuous
in the following cases:
\begin{equation*}
    1)~x\cdot x=0, \quad 2)~x\cdot 0=0, \quad 3)~0\cdot x=0, \quad
    4)~x\cdot a=0, \quad 5)~a\cdot x=0, \quad \mbox{for~}\; a\in
    B_\lambda\setminus\{ 0\}.
\end{equation*}
Then
\begin{equation*}
    U_n(x)\cdot U_n(x)=U_n(x)\cdot V(0)=V(0)\cdot U_n(x)=\{
    0\}\subseteq V(0),
\end{equation*}
for any $U_n(x)\in\mathscr{B}(x)$ and any open neighbourhood
$V(0)$ of zero in $S$ such that $V(0)\subseteq U(0)$. For every
$a\in B_\lambda\setminus\{ 0\}$ there exists a positive integer
$j$ such that
\begin{equation*}
\big(\varphi^{-1}(a\cdot a^{-1})\cup\psi^{-1}(a\cdot
a^{-1})\cup\varphi^{-1}(a^{-1}\cdot a)\cup\psi^{-1}(a^{-1}\cdot
a)\big)\cap A_j=\varnothing.
\end{equation*}
Then we have
\begin{equation*}
    \{ a\}\cdot U_j(x)=U_j(x)\cdot\{ x\}=\{ 0\}\subseteq V(0),
\end{equation*}
for any open neighbourhood $V(0)$ of zero in $S$ such that
$V(0)\subseteq U(0)$.

Obviously that $(B_\lambda,\tau)$ is a dense subsemigroup of the
topological semigroup $(S,\tau_0)$.
\end{proof}

\begin{theorem}\label{theorem1.6}
An infinite semigroup of matrix units does not embed into a
countably compact topological semigroup.
\end{theorem}

\begin{proof}
Suppose to the contrary that there exists a countably compact
topological semigroup $S$ which contains an infinite semigroup of
matrix units $B_\lambda$ for some infinite cardinal $\lambda$.
Since a closed subset of a countably compact space is countably
compact (see \cite[Theorem~3.10.4]{Engelking1989}), without loss
of generality we can assume that $B_\lambda$ is a dense
subsemigroup of $S$. Then by Lemma~\ref{lemma1.2}~$(iii)$, we have
that $E(S)=E(B_\lambda)$. Theorem~1.5~\cite[Vol.~1]{CHK} and
Theorem~3.10.4 of \cite{Engelking1989} implies that $E(B_\lambda)$
is a countably compact band of $S$. By Lemma~\ref{lemma1.1},
$E(B_\lambda)$ is compact and hence by Theorem~\ref{theorem1.3},
$B_\lambda$ is a closed subgroup of $S$. Therefore by
Theorem~3.10.4~\cite{Engelking1989}, $B_\lambda$ is a countably
compact topological semigroup. A contradiction to the fact that on
an infinite semigroup of matrix units there does not exist a
countably compact semigroup topology (see
\cite[Theorem~6]{GutikPavlyk2005}). The obtained contradiction
implies the assertion of the theorem.
\end{proof}

Since the semigroup of matrix units is congruence-free,
Theorem~\ref{theorem1.6} implies the following:

\begin{theorem}\label{theorem1.7}
Any continuous homomorphism of an infinite semigroup of matrix
units into a countably compact topological semigroup is
annihilating.
\end{theorem}

Since the semigroup of matrix units $B_\lambda$ is isomorphic to
the semigroup $\mathscr{I}_\lambda^{1}$ and
$\mathscr{I}_\lambda^{1}$ is a subsemigroup of
$\mathscr{I}_\lambda^{n}$ for all cardinals $\lambda\geqslant 1$,
Theorem~\ref{theorem1.7} implies the following:

\begin{corollary}\label{corollary1.9}
Let $\lambda$ be an infinite cardinal and $n$ be a positive
integer. Then there exists no a countably compact topological
semigroup $S$ which contains $\mathscr{I}_\lambda^{n}$.
\end{corollary}

\begin{question}\label{question1.10}
Is any $H$-closed semigroup topology on the infinite semigroup of
matrix units absolutely $H$-closed?
\end{question}

\begin{theorem}\label{theorem1.11}
Let $\lambda$ be an infinite cardinal and $n$ be a positive
integer. Then every continuous homomorphism of the topological
semigroup $\mathscr{I}_\lambda^{n}$ into a countably compact
topological semigroup is annihilating.
\end{theorem}

\begin{proof}
We shall prove the assertion of the theorem by induction. By
Theorem~\ref{theorem1.7} every continuous homomorphism of the
topological semigroup $\mathscr{I}_\lambda^{1}$ into a countably
compact topological semigroup $S$ is annihilating. We suppose that
the assertion of the theorem holds for $n=1,2,\ldots,k-1$ and we
shall prove that it is true for $n=k$.

Obviously it is sufficiently to show that the statement of the
theorem holds for the discrete semigroup
$\mathscr{I}_\lambda^{k}$. Let
$h\colon\mathscr{I}_\lambda^{k}\rightarrow S$ be arbitrary
homomorphism from $\mathscr{I}_\lambda^{k}$ with the discrete
topology into a countably compact topological semigroup $S$. Then
by Theorem~\ref{theorem1.7} the restriction
$h_{\mathscr{I}_\lambda^{1}}\colon\mathscr{I}_\lambda^{1}
\rightarrow S$ of homomorphism $h$ onto the subsemigroup
$\mathscr{I}_\lambda^{1}$ of $\mathscr{I}_\lambda^{k}$ is an
annihilating homomorphisms. Let
$(\mathscr{I}_\lambda^{1})h_{\mathscr{I}_\lambda^{1}}=
(\mathscr{I}_\lambda^{1})h=e$, where $e\in E(S)$. We fix any
$\alpha\in\mathscr{I}_\lambda^{k}$ with
$\operatorname{ran}(\alpha)=i\geqslant 2$. Let
 $\alpha=
 \left(%
 \begin{array}{cccc}
  x_1 & x_2 & \cdots & x_i \\
  y_1 & y_2 & \cdots & y_i \\
 \end{array}%
 \right)
 $
(where $x_1, x_2,\ldots, x_i, y_1, y_2,\ldots, y_i\in X$ for some
set $X$ of cardinality $\lambda$). We fix $y_1\in X$ and define
subsemigroup $T_{y_1}$ of $\mathscr{I}_\lambda^{k}$ as follows:
\begin{equation*}
    T_{y_1}=\left\{\beta\in\mathscr{I}_\lambda^{k}\mid
      \Big(%
     \begin{array}{c}
      y_1\\
      y_1\\
     \end{array}%
      \Big)\cdot\beta=\beta\cdot
      \Big(%
     \begin{array}{c}
      y_1\\
      y_1\\
     \end{array}%
      \Big)=
      \Big(%
     \begin{array}{c}
      y_1\\
      y_1\\
     \end{array}%
      \Big)
    \right\}.
\end{equation*}
Then the semigroup $T_{y_1}$ is isomorphic to the semigroup
$\mathscr{I}_\lambda^{k-1}$, the element
 $\Big(%
     \begin{array}{c}
      y_1\\
      y_1\\
     \end{array}%
      \Big)
 $
is zero of $T_{y_1}$ and hence by induction assumption we have
$\Big(\Big(%
     \begin{array}{c}
      y_1\\
      y_1\\
     \end{array}%
      \Big)\Big)h=(\beta)h
 $
for all $\beta\in T_{y_1}$.

Since
 $\Big(%
     \begin{array}{c}
      y_1\\
      y_1\\
     \end{array}%
      \Big)\in\mathscr{I}_\lambda^{1}
 $,
we have that $(\beta)h=(0)h$ for all $\beta\in T_{y_1}$. But
$\alpha=\alpha\gamma$, where
 $\gamma=
 \left(%
 \begin{array}{cccc}
  y_1 & y_2 & \cdots & y_i \\
  y_1 & y_2 & \cdots & y_i \\
 \end{array}%
 \right)\in T_{y_1}
 $,
and hence we have
\begin{equation*}
    (\alpha)h=(\alpha\gamma)h=(\alpha)h\cdot(\gamma)h=
    (\alpha)h\cdot(0)h=(\alpha\cdot 0)h=(0)h=e.
\end{equation*}
This completes the proof of the theorem.
\end{proof}

\begin{theorem}\label{theorem1.12}
An infinite semigroup of matrix units does not embed into a
Tychonoff topological semigroup with the pseudo-compact square.
\end{theorem}

\begin{proof}
Suppose to the contrary: there exists a Tychonoff topological
semigroup $S$ with the pseudo-compact square $S\times S$ which
contains the infinite semigroup of matrix units $B_\lambda$ for
some $\lambda\geqslant\omega$. By
Theorem~1.3~\cite{BanakhDimitrova2009xx} for any topological
semigroup $S$ with the pseudocompact square $S\times S$ the
semigroup operation $\mu\colon S\times S\rightarrow S$ extends to
a continuous semigroup operation $\beta\mu\colon\beta S\times\beta
S\rightarrow\beta S$ and the map $\beta\colon S\rightarrow\beta S$
is a homeomorphism ``into''. Therefore the restriction
$\beta|_{B_\lambda}\colon B_\lambda\rightarrow\beta S$ is an
embedding of a semigroup $B_\lambda$ into a compact topological
semigroup $\beta S$. This contradicts
Theorem~10~\cite{GutikPavlyk2005}. The obtained contradiction
implies the statement of the theorem.
\end{proof}

Since the semigroup of matrix units is congruence-free,
Theorem~\ref{theorem1.12} implies the following:

\begin{theorem}\label{theorem1.13}
Any continuous homomorphism of an infinite semigroup of matrix
units into a Tychonoff topological semigroup with the
pseudo-compact square is annihilating.
\end{theorem}

Since the semigroup of matrix units $B_\lambda$ is isomorphic to
the semigroup $\mathscr{I}_\lambda^{1}$ and
$\mathscr{I}_\lambda^{1}$ is a subsemigroup of
$\mathscr{I}_\lambda^{n}$ for all cardinals $\lambda\geqslant 1$,
Theorem~\ref{theorem1.13} implies the following:

\begin{corollary}\label{corollary1.14}
Let $\lambda$ be an infinite cardinal and $n$ be a positive
integer. Then there exists no a Tychonoff topological semigroup
$S$ with the pseudo-compact square which contains
$\mathscr{I}_\lambda^{n}$.
\end{corollary}

The proof of the following theorem is similar to
Theorem~\ref{theorem1.11}.

\begin{theorem}\label{theorem1.15}
Let $\lambda$ be an infinite cardinal and $n$ be a positive
integer. Then every continuous homomorphism of the topological
semigroup $\mathscr{I}_\lambda^{n}$ into a Tychonoff topological
semigroup with the pseudo-compact square is annihilating.
\end{theorem}


\section{$H$-closed semigroup topologies on
$\mathscr{I}_\lambda^{n}$}

Let $X$ be a set of cardinality $\lambda$. Fix any positive
integer $n$ and put
\begin{equation*}
\exp_n(\lambda)=\{ A\subseteq X\colon |A|\leqslant n\}.
\end{equation*}
Then $\exp_n(\lambda)$ with the operation ``$\cap$'' is a
semilattice. Further by $\exp_n(\lambda)$ we denote the
semilattice $(\exp_n(\lambda),\cap)$.

\begin{proposition}\label{proposition2.1}
For any cardinal $\lambda\geqslant 1$ and any positive integer $n$
the band $E(\mathscr{I}_\lambda^{n})$ of the semigroup
$\mathscr{I}_\lambda^{n}$ is isomorphic to the semilattice
$\exp_n(\lambda)$.
\end{proposition}

\begin{proof}
An isomorphism $h\colon
E(\mathscr{I}_\lambda^{n})\rightarrow\exp_n(\lambda)$ we define by
the formula: $h(\alpha)=\operatorname{dom}\alpha$.
\end{proof}

An element $e$ of a topological semilattice $E$ is called a
\emph{local minimum} if there exists an open neighbourhood $U(e)$
of $e$ such that
$U(e)\cap{\downarrow}e\subseteq{\uparrow}e$~\cite{HofmannMisloveStralka}.

\begin{lemma}\label{lemma2.2}
Let $n$ be a positive integer and $\lambda\geqslant 1$. If $\tau$
is a semigroup topology on $\exp_n(\lambda)$, then any idempotent
of $\exp_n(\lambda)$ is a local minimum and hence ${\uparrow}e$ is
an open subset of $(\exp_n(\lambda),\tau)$.
\end{lemma}

\begin{proof}
We observe that the statement of the lemma is trivial in case when
$e$ is the zero of the semigroup $\exp_n(\lambda)$. We suppose
that $|e|=k$ for some $k=1,2,\ldots,n$. Since the set
$\exp_{k-1}(\lambda)$ is a subsemilattice of $\exp_n(\lambda)$,
Theorem~9 of \cite{Stepp1975} implies that the set
$U(e)=\exp_n(\lambda)\setminus\exp_{k-1}(\lambda)$ is an open
neighbourhood of $e$ in $\exp_n(\lambda)$. This implies that $e$
is a local minimum in $\exp_n(\lambda)$. The continuity of
semilattice operation in $\exp_n(\lambda)$ implies that there
exists an open neighbourhood $V(e)$ of $e$ such that $V(e)\cdot
e\subseteq U(e)$. This implies that $V(e)\subseteq{\uparrow}e$.
Then Theorem~VI-1.13$(iii)$ of
\cite{GierzHofmannKeimelLawsonMisloveScott2003} implies that
${\uparrow}e$ is an open subset of $(\exp_n(\lambda),\tau)$.
\end{proof}

We define the family $\mathscr{B}$ of non-empty subsets in
$\exp_n(\lambda)$ as follows:
\begin{equation}\label{base{mathcsr{B}}}
\begin{split}
    \mathscr{B}=\{U(e;e_1,\ldots,e_i)=
        & \;{\uparrow}e
    \setminus({\uparrow}e_1\cup\cdots\cup{\uparrow}e_i)\mid
    e,e_1,\ldots,e_i\in\exp_n(\lambda) \\
        & \mbox{~such that~} \;  e<e_1,\ldots,e<e_i, \; i\in\mathbb{N}\}.
\end{split}
\end{equation}

\begin{proposition}\label{proposition2.3}
The topology $\tau_c$ generated by the base $\mathscr{B}$ is the
unique compact Hausdorff topology on $\exp_n(\lambda)$ such that
$(\exp_n(\lambda),\tau_c)$ is a topological semilattice.
\end{proposition}

\begin{proof}
Let $e,f\in\exp_n(\lambda)$. If $e\leqslant f$ then $V(e)\cap
V(f)=\varnothing$ for open neighbourhoods $V(e)=U(e;f)$ and
$V(f)={\uparrow}f$ of $e$ and $f$, respectively. If the
idempotents $e$ and $f$ are incomparable, then we put $g=e\cup f$
and hence we have that $U(e;g)\cap U(f;g)=\varnothing$. Therefore
$\tau_c$ is a Hausdorff topology on $\exp_n(\lambda)$. We observe
that since any chain in $\exp_n(\lambda)$ has $\leqslant n$
elements, $\tau_c$ is a compact topology on $\exp_n(\lambda)$.

Next we show that the semilattice operation $\cap$ on
$(\exp_n(\lambda),\tau_c)$ is continuous. Let $e$ and $f$ be
arbitrary elements from $\exp_n(\lambda)$. If $e=f$ then
\begin{equation*}
 U(e;e_1,\ldots,e_k)\cdot U(e;e_1,\ldots,e_k)\subseteq
 U(e;e_1,\ldots,e_k)
\end{equation*}
for all $U(e;e_1,\ldots,e_k)\in\mathscr{B}$. If $e<f$ and
$U(e;e_1,\ldots,e_k)\in\mathscr{B}$ is an open neighbourhood of
$e$, then
\begin{equation*}
    U(e;e_1,\ldots,e_k,f)\cdot U(f;f_1,\ldots, f_m)\subseteq
    U(e;e_1,\ldots,e_k)
\end{equation*}
for all $f_1,\ldots,f_m\in{\uparrow}f\setminus\{ f\}$. If the
idempotents $e$ and $f$ are incomparable, then we have $g=e\cap
f<e, f$. Put $h=e\cup f$. Then for any open neighbourhood
$U(g;g_1,\ldots,g_k)\in\mathscr{B}$ of $g$ we have
\begin{equation*}
    U(e;h)\cdot U(f;h)=\{ g\}\subseteq U(g;g_1,\ldots,g_k).
\end{equation*}
Hence $(\exp_n(\lambda),\tau_c)$ is a topological semilattice.

The uniqueness of the topology $\tau_c$ follows from
Lemma~\ref{lemma2.2}.
\end{proof}

\begin{proposition}\label{proposition2.3a}
For a topological semilattice $\exp_n(\lambda)$ the following
conditions are equivalent:
\begin{enumerate}
    \item[$(i)$] $\exp_n(\lambda)$ is compact;
    \item[$(ii)$] $\exp_n(\lambda)$ is countably compact.
\end{enumerate}
\end{proposition}

\begin{proof}
Since the statement of the proposition is trivial when the
semilattice $\exp_n(\lambda)$ is finite, we suppose that the
cardinal $\lambda$ is infinite.

We observe that the implication $(i)\Rightarrow(ii)$ is trivial.

We shall prove the implication $(ii)\Rightarrow(i)$ by induction.
By Lemma~\ref{lemma1.1} every countably compact semigroup topology
on $\exp_1(\lambda)$ is compact. We suppose that the assertion of
the proposition holds for $k=1,2,\ldots,n-1$ and we shall prove
that it is true for $k=n$. Suppose the contrary: there exists a
semigroup topology $\tau$ on $\exp_n(\lambda)$ such that
$(\exp_n(\lambda),\tau)$ is a countably compact non-compact
topological semilattice. Let
$\mathscr{C}=\{U_\alpha\}_{\alpha\in\mathscr{D}}$ be an open cover
of the topological semilattice $(\exp_n(\lambda),\tau)$ which does
not contain a finite subcover. Since by assumption of induction
the subsemilattice $\exp_{n-1}(\lambda)$ is compact, there exists
a finite subfamily
$\mathscr{C}_0=\{U_{\alpha_1},\ldots,U_{\alpha_n}\}_{\alpha_1,
\ldots,\alpha_n\in\mathscr{D}}$ of $\mathscr{C}$ such that
$\mathscr{C}_0$ is an open cover of $\exp_{n-1}(\lambda)$. Since
the topological space $(\exp_n(\lambda),\tau)$ is non-compact and
by Lemma~\ref{lemma2.2} any idempotent
$\varepsilon\in\exp_n(\lambda)\setminus\exp_{n-1}(\lambda)$ is an
isolated point in $(\exp_n(\lambda),\tau)$, we have that
$A=\exp_n(\lambda)\setminus(U_{\alpha_1}\cup\cdots\cup
U_{\alpha_n})$ is a closed discrete infinite subspace of
$(\exp_n(\lambda),\tau)$. But Theorem~3.10.4 of
\cite{Engelking1989} implies that $A$ is a countably compact
space, a contradiction. The obtained contradiction implies the
statement of the proposition.
\end{proof}

The following example shows that there exists a semigroup topology
$\tau^c$ on the semigroup $\mathscr{I}_\lambda^{n}$ such that
$E(\mathscr{I}_\lambda^{n})$ is a compact semilattice.

\begin{example}\label{example2.4}
We identify the semilattice $E(\mathscr{I}_\lambda^{n})$ with the
semilattice $\exp_n(\lambda)$. Let $\tau_c$ be the topology on
$E(\mathscr{I}_\lambda^{n})$ determined by the base $\mathscr{B}$
(see: (\ref{base{mathcsr{B}}})). We define the topology $\tau^c$
on $\mathscr{I}_\lambda^{n}$ as follows:
\begin{itemize}
    \item[$(i)$] at any idempotent $e\in\mathscr{I}_\lambda^{n}$
                 the base of the topology $\tau^c$ coincides with
                 the base of $\tau_c$;
    \item[$(ii)$] all non-idempotent elements of the semigroup
                 $\mathscr{I}_\lambda^{n}$ are isolated points.
\end{itemize}
\end{example}

\begin{proposition}\label{proposition2.5}
$(\mathscr{I}_\lambda^{n},\tau^c)$ is a topological inverse
semigroup.
\end{proposition}

\begin{proof}
Since all non-idempotent elements of the semigroup
$\mathscr{I}_\lambda^{n}$ are isolated points in
$(\mathscr{I}_\lambda^{n},\tau^c)$,
Proposition~\ref{proposition2.3} implies that it is completely to
show that the semigroup operation in
$(\mathscr{I}_\lambda^{n},\tau^c)$ is continuous in the following
two cases:
\begin{equation*}
    \mbox{a)} \quad \alpha\cdot\varepsilon; \qquad \mbox{and}
    \qquad \mbox{b)} \quad \varepsilon\cdot\alpha, \qquad
    \mbox{for} \quad \alpha\in\mathscr{I}_\lambda^{n}\setminus
    E(\mathscr{I}_\lambda^{n}) \quad \mbox{and} \quad
    \varepsilon\in E(\mathscr{I}_\lambda^{n}).
\end{equation*}

By Corollary~8~\cite{GutikLawsonRepovs2009} an idempotent
$\varepsilon$ of the semigroup $\mathscr{I}_\lambda^{n}$ is an
isolated point in $\mathscr{I}_\lambda^{n}$ in the case
$\operatorname{rank}\alpha=n$. Therefore we can assume that
$\operatorname{rank}\varepsilon<n$.

In case a) let $\{x_1,\ldots,x_s\}=\{
x\in\operatorname{ran}\alpha\mid
x\notin\operatorname{dom}\varepsilon\}$. Put
$\varepsilon_1=\varepsilon\cup\{
x_1\},\ldots,\varepsilon_s=\varepsilon\cup\{ x_s\}$. Then
\begin{equation*}
    \{\alpha\}\cdot
    U(\varepsilon;\varepsilon_1,\ldots,\varepsilon_s)=\{\alpha\cdot\varepsilon\}.
\end{equation*}
In case b) let $\{y_1,\ldots,y_p\}=\{
x\in\operatorname{dom}\alpha\mid
x\notin\operatorname{ran}\varepsilon\}$. Put
$\varepsilon_1=\varepsilon\cup\{
y_1\},\ldots,\varepsilon_p=\varepsilon\cup\{ y_p\}$. Then
\begin{equation*}
    U(\varepsilon;\varepsilon_1,\ldots,\varepsilon_p)\cdot\{\alpha\}
    =\{\varepsilon\cdot\alpha\}.
\end{equation*}
Therefore the semigroup operation is continuous on
$(\mathscr{I}_\lambda^{n},\tau^c)$.

The continuity of the inversion in
$(\mathscr{I}_\lambda^{n},\tau^c)$ follows from the facts that the
band $E(\mathscr{I}_\lambda^{n})$ is a compact open subsemigroup
of $(\mathscr{I}_\lambda^{n},\tau^c)$ and all non-idempotent
elements of the semigroup $\mathscr{I}_\lambda^{n}$ are isolated
points in $(\mathscr{I}_\lambda^{n},\tau^c)$.
\end{proof}

\begin{theorem}\label{theorem2.6}
Let $(\mathscr{I}_\lambda^{n},\tau)$ be a topological semigroup.
If the band $E(\mathscr{I}_\lambda^{n})$ is a compact subset of
$(\mathscr{I}_\lambda^{n},\tau)$, then
$(\mathscr{I}_\lambda^{n},\tau)$ is an absolutely $H$-closed
topological semigroup.
\end{theorem}

\begin{proof}
We shall prove the assertion of the theorem by induction. By
Theorem~\ref{theorem1.4} the topological semigroup
$\mathscr{I}_\lambda^{1}$ with compact band is absolutely
$H$-closed. We suppose that the assertion of the theorem holds for
$n=1,2,\ldots,k-1$ and we shall prove that it is true for $n=k$.

Suppose the contrary: the topological semigroup
$(\mathscr{I}_\lambda^{n},\tau)$ is not absolutely $H$-closed.
Then there exist a Hausdorff topological semigroup $T$ and
continuous homomorphism $h\colon
\mathscr{I}_\lambda^{n}\rightarrow T$ such that
$(\mathscr{I}_\lambda^{n})h$ is not a closed subsemigroup of $T$.
Since the closure $\operatorname{cl}_G(L)$ of a subsemigroup $L$
of topological semigroup $G$ is a subsemigroup in $G$ (see
\cite[Vol.~1, p.~9]{CHK}), without loss of generality we cam
assume that $(\mathscr{I}_\lambda^{n})h$ is a dense subsemigroup
of $T$ and $T\setminus(\mathscr{I}_\lambda^{n})h\neq\varnothing$.
Then by Lemma~8~\cite{GutikPavlyk2005}, zero of the semigroup
$(\mathscr{I}_\lambda^{n})h$ is zero in $T$ and we denote it by
$0_T$. Let $x\in T\setminus(\mathscr{I}_\lambda^{n})h$. Then
$0_T\cdot x=0_T$. The continuity of the semigroup operation in $T$
implies that for any open neighbourhood $W(0_T)$ of zero $0_T$ in
$T$ there exist open neighbourhoods $U(0_T)$ and $V(0_T)$ of $0_T$
in $T$ and an open neighbourhood $V(x)$ of $x$ in $T$  such that
\begin{equation*}
    V(0_T)\cdot V(x)\subseteq U(0_T)\subseteq W(0_T), \qquad
    V(0_T)\subseteq U(0_T) \qquad \mbox{and} \qquad U(0_T)\cap
    V(x)=\varnothing.
\end{equation*}
By the assumption of the induction and Theorem~9 of
\cite{Stepp1975} we have that
$V(x)\cap\big((\mathscr{I}_\lambda^{n-1})h\cup
(E(\mathscr{I}_\lambda^{n}))h\big)=\varnothing$. We remark that
for any idempotent $\varepsilon_0$ of the semigroup
$\mathscr{I}_\lambda^{n}$ with $\operatorname{ran}\varepsilon_0=1$
the set $\mathscr{I}_\lambda^{n}(\varepsilon_0)=
\{\chi\in\mathscr{I}_\lambda^{n}\mid
\chi\varepsilon_0=\varepsilon_0\chi=\varepsilon_0\}$ is a
subsemigroup of $\mathscr{I}_\lambda^{n}$ and simple observations
show that $\mathscr{I}_\lambda^{n}(\varepsilon_0)$ is
algebraically isomorphic to the semigroup
$\mathscr{I}_\lambda^{n-1}$. Then the compactness of the band
$E(\mathscr{I}_\lambda^{n})$ implies that there exist a finite
subset of idempotents $\{\varepsilon_1,\ldots,\varepsilon_k\}$ in
$\mathscr{I}_\lambda^{n}$ with
$\operatorname{ran}\varepsilon_1=\ldots
=\operatorname{ran}\varepsilon_k=1$, an open neighbourhood
$\widetilde{V}(0)$ of the zero $0$ of the semigroup
$\mathscr{I}_\lambda^{n}$ and an open neighbourhood $O(x)$ of $x$
in $T$ such that
\begin{equation*}
\begin{split}
    (\widetilde{V}(0))h\subseteq & \;V(0_T),  \quad
    \widetilde{V}(0)\cap\big(\mathscr{I}_\lambda^{n}(\varepsilon_1)
    \cup\ldots\mathscr{I}_\lambda^{n}(\varepsilon_k)\big)=
    \varnothing, \quad O(x)\subseteq V(x) \quad
    \mbox{and} \quad \\
    &\big(O(x)\cap(\mathscr{I}_\lambda^{n})h\big)h^{-1}\cap
    \big(\mathscr{I}_\lambda^{n}(\varepsilon_1)
    \cup\ldots\mathscr{I}_\lambda^{n}(\varepsilon_k)\big)=
    \varnothing.
\end{split}
\end{equation*}
The compactness of the band $E(\mathscr{I}_\lambda^{n})$ and the
infiniteness of the set
$(O(x)\cap\big(\mathscr{I}_\lambda^{n})h\big)h^{-1}$ imply that
there exists
$\alpha\in\big(O(x)\cap(\mathscr{I}_\lambda^{n})h\big)h^{-1}$ such
that $\alpha\in \widetilde{V}(0)\cdot\alpha$. Then
\begin{equation*}
    \big(\widetilde{V}(0)\big)h\cdot(\alpha)h\subseteq V(0_T)\cdot
    O(x)\subseteq U(0_T) \quad \mbox{and} \quad
    \big((\widetilde{V}(0))h\cdot(\alpha)h\big)\cap
    O(x)\neq\varnothing.
\end{equation*}
This contradicts to the assumption $U(0_T)\cap V(x)=\varnothing$.
The obtained contradiction implies that
$T\setminus(\mathscr{I}_\lambda^{n})h=\varnothing$, and hence
$\mathscr{I}_\lambda^{n}$ is an absolutely $H$-closed topological
semigroup.
\end{proof}

Proposition~\ref{proposition2.3a} and Theorem~\ref{theorem2.6}
imply the following:

\begin{corollary}\label{corollary2.6a}
Let $(\mathscr{I}_\lambda^{n},\tau)$ be a topological semigroup.
If the band $E(\mathscr{I}_\lambda^{n})$ is a countably compact
subset of $(\mathscr{I}_\lambda^{n},\tau)$, then
$(\mathscr{I}_\lambda^{n},\tau)$ is an absolutely $H$-closed
topological semigroup.
\end{corollary}


\section{Countably compact topological Brandt
$\lambda^0$-extensions}

\begin{definition}[\cite{GutikPavlyk2006}]\label{definition4-1}
Let $\mathscr{S}$ be some class of topological monoids with zero.
Let $\lambda$ be any cardinal $\geqslant 1$, and
$(S,\tau)\in\mathscr{S}$. Let $\tau_{B}$ be a topology on
$B^0_{\lambda}(S)$ such that
\begin{itemize}
  \item[a)] $\left(B^0_{\lambda}(S), \tau_{B}\right)\in\mathscr{S};$
  \item[b)] $\tau_{B}|_{S_{\alpha,\alpha}}=\tau$ for some
$\alpha\in\lambda$.
\end{itemize}
Then $\left(B^0_{\lambda}(S), \tau_{B}\right)$ is called a
\emph{topological Brandt $\lambda^0$-extension of $(S, \tau)$ in}
$\mathscr{S}$. If $\mathscr{S}$ coincides with the class of all
topological semigroups, then $\left(B^0_{\lambda}(S),
\tau_{B}\right)$ is called a \emph{topological Brandt
$\lambda^0$-extension of} $(S, \tau)$.
\end{definition}

A topological Brandt $\lambda^0$-extension
$\left(B^0_{\lambda}(S), \tau_{B}\right)$ is called \emph{compact}
(resp., \emph{countably compact}) if the topological space
$\left(B^0_{\lambda}(S), \tau_{B}\right)$ is compact (resp.,
countably compact)~\cite{GutikRepovs}. Gutik and Repov\v{s} in
\cite{GutikRepovs} described the structures of compact topological
Brandt $\lambda^0$-extensions.

We need the following lemma from~\cite{GutikRepovs}:

\begin{lemma}[\cite{GutikRepovs}]\label{lemma4-1a}
For any topological monoid $(S,\tau)$ with zero and for any finite
cardinal $\lambda\geqslant 1$ there exists an unique topological
Brandt $\lambda^0$-extension $\left(B^0_{\lambda}(S),
\tau_{B}\right)$ and the topology $\tau_{B}$ generated by the base
$\mathscr{B}_B=\bigcup\{\mathscr{B}_B(t)\mid t\in
B^0_{\lambda}(S)\}$, where:
\begin{itemize}
    \item[$(i)$] $\mathscr{B}_B(t)=\{(U(s))_{\alpha,\beta}
    \setminus\{ 0_S\}\mid U(s)\in\mathscr{B}_S(s)\}$, when
    $t=(\alpha,s,\beta)$ is a non-zero element of
    $B^0_{\lambda}(S)$, $\alpha,\beta\in\lambda$;
    \item[$(ii)$] $\mathscr{B}_B(0)=\{\bigcup_{\alpha,\beta\in
    \lambda}(U(0_S))_{\alpha,\beta}\mid U(0_S)\in\mathscr{B}_S(0_S)\}$,
    when $0$ is the zero of $B^0_{\lambda}(S)$,
\end{itemize}
and $\mathscr{B}_S(s)$ is a base of the topology $\tau$ at the
point $s\in S$.
\end{lemma}

Proposition~\ref{proposition4-2} describes the structures of
countably compact Brandt $\lambda^0$-extensions of topological
monoids.

\begin{proposition}\label{proposition4-2}
A topological Brandt $\lambda^0$-extension $B^0_{\lambda}(S)$ of a
topological monoid $(S,\tau)$ with zero is countably compact if
and only if the cardinal $\lambda\geqslant 1$ is finite and
$(S,\tau)$ is a countably compact topological semigroup. Moreover,
for any countably compact topological monoid $(S,\tau)$ with zero
and for any finite cardinal $\lambda\geqslant 1$ there exists an
unique countably compact topological Brandt $\lambda^0$-extension
$\left(B^0_{\lambda}(S), \tau_{B}\right)$ and the topology
$\tau_{B}$ generated by the base
$\mathscr{B}_B=\bigcup\{\mathscr{B}_B(t)\mid t\in
B^0_{\lambda}(S)\}$, where:
\begin{itemize}
    \item[$(i)$] $\mathscr{B}_B(t)=\{(U(s))_{\alpha,\beta}
    \setminus\{ 0_S\}\mid U(s)\in\mathscr{B}_S(s)\}$, when
    $t=(\alpha,s,\beta)$ is a non-zero element of
    $B^0_{\lambda}(S)$, $\alpha,\beta\in\lambda$;
    \item[$(ii)$] $\mathscr{B}_B(0)=\{\bigcup_{\alpha,\beta\in
    \lambda}
    (U(0_S))_{\alpha,\beta}\mid U(0_S)\in\mathscr{B}_S(0_S)\}$,
    when $0$ is the zero of $B^0_{\lambda}(S)$,
\end{itemize}
and $\mathscr{B}_S(s)$ is a base of the topology $\tau$ at the
point $s\in S$.
\end{proposition}

\begin{proof}
Since by Theorem~\ref{theorem1.7} the infinite semigroup of matrix
units does not embed into a countably compact topological
semigroup, the countable compactness of the topological Brandt
$\lambda^0$-extension $\left(B^0_{\lambda}(S), \tau_{B}\right)$ of
a topological semigroup $(S,\tau)$ implies that the cardinal
$\lambda$ is finite. Then by Theorem~1.7(e) of~\cite[Vol.~1]{CHK},
$(\alpha,1_S,\alpha)B^0_{\lambda}(S)(\alpha,1_S,\alpha)=S_{\alpha,\alpha}$
is a countably compact semigroup for any $\alpha\in\lambda$, and
hence $(S,\tau)$ is a countably compact topological semigroup. The
converse follows from Lemma~1~\cite{GutikPavlyk2006} and the
assertion that the finite union of countably compact spaces is a
countably compact space.

Lemma~\ref{lemma4-1a} implies the last assertion of the
proposition.
\end{proof}

Lemma~\ref{lemma4-1a} and Proposition~\ref{proposition4-2} imply

\begin{theorem}\label{theorem4.2a}
Every countably compact (and hence compact) topological Brandt
$\lambda^0$-extension \break $\left(B^0_{\lambda}(S),
\tau_{B}\right)$ of a topological inverse semigroup $(S,\tau)$ is
a topological inverse semigroup.
\end{theorem}

\begin{definition}[\cite{GutikRepovs}]\label{definition4-3}
Let $\lambda$ be any cardinal $\geqslant 2$. We shall say that a
semigroup $S$ has the \emph{$\mathcal{B}_{\lambda}^*$-pro\-perty}
if $S$ satisfies the following conditions:
\begin{itemize}
    \item[1)] $T$ does not contain the semigroup of
    $\lambda\times\lambda$-matrix units;
    \item[2)] $T$ does not contain the semigroup of $2\times 2$-matrix
    units $B_2$ such that the zero of $B_2$ is the zero of $T$.
\end{itemize}
\end{definition}

Gutik and Repov\v{s} in \cite{GutikRepovs} proved the following:

\begin{theorem}[\cite{GutikRepovs}]\label{theorem4-4}
Let $\lambda_1$ and $\lambda_2$ be any finite cardinals such that
$\lambda_2\geqslant\lambda_1\geqslant 1$. Let $B_{\lambda_1}^0(S)$
and $B_{\lambda_2}^0(T)$ be topological Brandt $\lambda_1^0$- and
$\lambda_2^0$-extensions of topological monoids $S$ and $T$ with
zero, respectively. Let $h\colon S\rightarrow T$ be a continuous
homomorphism such that $(0_S)h=0_T$ and $\varphi\colon
{\lambda_1}\rightarrow{\lambda_2}$ an one-to-one map. Let $e$ be a
non-zero idempotent of $T$, $H_e$ a maximal subgroup of $T$ with
unity $e$ and $u\colon{\lambda_1}\rightarrow H_e$ a map. Then
$I_h=\{ s\in S\mid (s)h=0_T\}$ is a closed ideal of $S$ and the
map $\sigma\colon B_{\lambda_1}^0(S)\rightarrow
B_{\lambda_2}^0(T)$ defined by the formulae
\begin{equation*}
    ((\alpha,s,\beta))\sigma=
    \left\{%
\begin{array}{cl}
    ((\alpha)\varphi,(\alpha)u\cdot(s)h\cdot((\beta)u)^{-1},(\beta)\varphi),
    & \hbox{if}\quad s\notin S\setminus I_h ;\\
    0_2, & \hbox{if}\quad s\in I_h^*,\\
\end{array}%
\right.
\end{equation*}
and $(0_1)\sigma=0_2$, is a non-trivial continuous homomorphism
from $B_{\lambda_1}^0(S)$ into $B_{\lambda_2}^0(T)$. Moreover if
for the semigroup $T$ the following conditions hold:
\begin{itemize}
    \item[($i$)] every idempotent of $T$ lies in the center of
    $T$;
    \item[($ii$)] $T$ has $\mathcal{B}_{\lambda_1}^*$-property,
\end{itemize}
then every non-trivial continuous homomorphism from
$B_{\lambda_1}^0(S)$ into $B_{\lambda_2}^0(T)$ can be so
constructed.
\end{theorem}

Next we define a category of countably compact topological monoids
and pairs of finite sets and a category of countably compact
topological semigroups.

Let $S$ and $T$ be countably compact topological monoids with
zeros. Let $\texttt{CHom}\,_0(S,T)$ be a set of all continuous
homomorphisms $\sigma\colon S\rightarrow T$ such that
$(0_S)\sigma=0_T$. We put
\begin{equation*}
    \mathbf{E}^{\textit{top}}_1(S,T)=\{e\in E(T)\mid \hbox{there exists~}
    \sigma\in
    \texttt{CHom}\,_0(S,T) \hbox{~such that~} (1_S)\sigma=e\}
\end{equation*}
and define the family
\begin{equation*}
    \mathscr{H}^{\textit{top}}_1(S,T)=\{ H(e)\mid
    e\in\mathbf{E}^{\textit{top}}_1(S,T)\},
\end{equation*}
where by $H(e)$ we denote the maximal subgroup with the unity $e$
in the semigroup $T$. Also by $\mathfrak{CCTB}$ we denote the
class of all countably compact topological monoids $S$ with zero
such that $S$ has $\mathcal{B}^*$-property and every idempotent of
$S$ lies in the center of $S$.

We define a category $\mathscr{TCC\!B}_{\operatorname{fin}}$ as
follows:
\begin{itemize}
    \item[$(i)$] $\operatorname{\textbf{Ob}}(\mathscr{TCC\!B}_{\operatorname{fin}})=
    \{(S,I)\mid S\in\mathfrak{CCTB} \textrm{~and~} I \textrm{~is a
    finite set}\}$, and if $S$ is a trivial
    semigroup then we identify $(S,I)$ and $(S,J)$ for all finite
    sets $I$ and $J$;
    \item[$(ii)$] $\operatorname{\textbf{Mor}}(\mathscr{TCC\!B}_{\operatorname{fin}})$
    consists of triples $(h,u,\varphi)\colon
    (S,I)\rightarrow(S^{\,\prime},I^{\,\prime})$, where
\begin{equation*}\label{eq4-1}
\begin{split}
    & h\colon S\rightarrow S^{\,\prime} \textrm{~is a continuous
      homomorphism such that~}
      h\in \texttt{CHom}\,_0(S,S^{\,\prime}),\\
    & u\colon I\rightarrow H(e) \textrm{~is a map~}, \textrm{~for~} H(e)\in
      \mathscr{H}^{\textit{top}}_1(S,S^{\,\prime}),\\
    & \varphi\colon I\rightarrow I^{\,\prime} \textrm{~is an
    one-to-one map},
\end{split}
\end{equation*}
with the composition
\begin{equation*}\label{eq3-2}
    (h,u,\varphi)(h^{\,\prime},u^{\,\prime},\varphi^{\,\prime})=
    (hh^{\,\prime},[u,\varphi,h^{\,\prime},u^{\,\prime}],\varphi\varphi^{\,\prime}),
\end{equation*}
where the map $[u,\varphi,h^{\,\prime},u^{\,\prime}]\colon
I\rightarrow H(e)$ is defined by the formula
\begin{equation*}\label{eq3-3}
    (\alpha)[u,\varphi,h^{\,\prime},u^{\,\prime}]=
    ((\alpha)\varphi)u^{\,\prime}\cdot((\alpha)u)h^{\,\prime}
    \qquad \textrm{~for~} \alpha\in I.
\end{equation*}
\end{itemize}
Straightforward verification shows that
$\mathscr{TCC\!B}_{\operatorname{fin}}$ is the category with the
identity morphism
$\varepsilon_{(S,I)}=(\operatorname{Id}_S,u_0,\operatorname{Id}_I)$
for any
$(S,I)\in\operatorname{\textbf{Ob}}(\mathscr{TCC\!B}_{\operatorname{fin}})$,
where $\operatorname{Id}_S\colon S\rightarrow S$ and
$\operatorname{Id}_I\colon I\rightarrow I$ are identity maps and
$(\alpha)u_0=1_S$ for all $\alpha\in I$.

We define a category
$\mathscr{B}^*_{\operatorname{fin}}(\mathscr{TCC\!S})$ as follows:
\begin{itemize}
    \item[$(i)$]
    let
    $\operatorname{\textbf{Ob}}(\mathscr{B}^*_{\operatorname{fin}}(\mathscr{TCC\!S}))$
    be all finite topological Brandt $\lambda^0$-extensions of
    countably topological monoids $S$ with zeros
    such that $S$ has $\mathcal{B}^*$-property and every
    idempotent of $S$ lies in the center of $S$;
    \item[$(ii)$]
    let
    $\operatorname{\textbf{Mor}}(\mathscr{B}^*_{\operatorname{fin}}(\mathscr{TCC\!S}))$
    be homomorphisms of finite topological Brandt
    $\lambda^0$-extensions of countably compact topological
    monoids $S$ with zeros such that $S$ has
    $\mathcal{B}^*$-property and every idempotent of $S$ lies in
    the center of $S$.
\end{itemize}

For each
$(S,I_{\lambda_1})\in\operatorname{\textbf{Ob}}(\mathscr{TCC\!B}_{\operatorname{fin}})$
with non-trivial $S$, let
$\textbf{B}(S,I_{\lambda_1})=B_{\lambda_1}^0(S)$ be the countably
compact topological Brandt $\lambda^0$-extension of the countably
compact topological monoid $S$. For each
$(h,u,\varphi)\in\operatorname{\textbf{Mor}}(\mathscr{TCC\!B}_{\operatorname{fin}})$
with a non-trivial continuous homomorphism $h$, where
$(h,u,\varphi)\colon$
$(S,I_{\lambda_1})\rightarrow(T,I_{\lambda_2})$ and
$(T,I_{\lambda_2})\in\operatorname{\textbf{Ob}}(\mathscr{TCC\!B}_{\operatorname{fin}})$,
we define a map $\textbf{B}{(h,u,\varphi)}\colon
\textbf{B}(S,I_{\lambda_1})=B_{\lambda_1}^0(S)\rightarrow
\textbf{B}(T,I_{\lambda_2})=B_{\lambda_2}^0(T)$ as follows:
\begin{equation*}\label{functor_B}
    ((\alpha,s,\beta))[\textbf{B}{(h,u,\varphi)}]=
    \left\{%
\begin{array}{cl}
    ((\alpha)\varphi,(\alpha)u\cdot(s)h\cdot((\beta)u)^{-1},(\beta)\varphi),
    & \hbox{if}\quad s\notin S\setminus I_h ;\\
    0_2, & \hbox{if}\quad s\in I_h^*,\\
\end{array}%
\right.
\end{equation*}
and $(0_1)[\textbf{B}{(h,u,\varphi)}]=0_2$, where $I_h=\{ s\in
S\mid (s)h=0_T\}$ is a closed ideal of $S$ and $0_1$ and $0_2$ are
the zeros of the semigroups $B_{\lambda_1}^0(S)$ and
$B_{\lambda_2}^0(T)$, respectively. For each
$(h,u,\varphi)\in\operatorname{\textbf{Mor}}(\mathscr{TCC\!B}_{\operatorname{fin}})$
with a trivial homomorphism $h$ we define a map
$\textbf{B}{(h,u,\varphi)}\colon
\textbf{B}(S,I_{\lambda_1})=B_{\lambda_1}^0(S)\rightarrow
\textbf{B}(T,I_{\lambda_2})=B_{\lambda_2}^0(T)$ as follows:
$(a)[\textbf{B}{(h,u,\varphi)}]=0_2$ for all $a\in
\textbf{B}(S,I_{\lambda_1})=B_{\lambda_1}^0(S)$. If $S$ is a
trivial semigroup then we define $\textbf{B}(S,I_{\lambda_1})$ to
be a trivial semigroup.

A functor $\textbf{F}$ from a category $\mathscr{C}$ into a
category $\mathscr{K}$ is called \emph{full} if for any
$a,b\in\operatorname{\textbf{Ob}}(\mathscr{C})$ and for any
$\mathscr{K}$-morphism $\alpha\colon \textbf{F}a\rightarrow
\textbf{F}b$ there exists a $\mathscr{C}$-morphism $\beta\colon
a\rightarrow b$ such that $\textbf{F}\beta=\alpha$, and
$\textbf{F}$ called \emph{representative} if for any
$a\in\operatorname{\textbf{Ob}}(\mathscr{K})$ there exists
$b\in\operatorname{\textbf{Ob}}(\mathscr{C})$ such that $a$ and
$\textbf{F}b$ are isomorphic.

Theorem~4.1 \cite{GutikRepovs} and Theorem~\ref{theorem4-4} imply

\begin{theorem}\label{theorem4-15}
$\operatorname{\textbf{B}}$ is a full representative functor from
$\mathscr{TCC\!B}_{\operatorname{fin}}$ into
$\mathscr{B}^*_{\operatorname{fin}}(\mathscr{TCC\!S})$.
\end{theorem}

\begin{remark}
We observe that the similar statements to
Theorem~\ref{theorem4-15} hold for the categories of countably
compact topological inverse monoids, countably compact Clifford
topological inverse monoids, countably compact Brandt topological
semigroups, countably compact topological semilattices and finite
sets and corresponding their countably compact topological Brandt
$\lambda^0$-extensions of countably compact topological monoids.
Moreover in the case of countably compact topological semilattices
the functor $\operatorname{\textbf{B}}$ determines the equivalency
of such categories. The last assertion follows from
Proposition~4.3~\cite{GutikRepovs}.
\end{remark}

\section*{Acknowledgements}

We thank Taras Banakh and Igor Guran for several comments. Also,
we thank the referee for several comments and remarks.



\begin{thebibliography}{99}

\bibitem{Abd-AllahBrown1980} A.~Abd-Allah and R.~Brown, {\it A
compact-open topology on partial maps with open domains,} J.
London Math. Soc. {\bf 21} (2) (1980), 480---486.

\bibitem{Baird1977} B.~B.~Baird, {\it Inverse semigroups of
homeomorphisms between open subsets,} J. Austral. Math. Soc. Ser.
A {\bf 24}:1 (1977), 92---102.

\bibitem{Baird1977a} B.~B.~Baird, {\it Embedding inverse
semigroups of homeomorphisms on closed subsets,} Glasgow Math. J.
{\bf 18}:2 (1977), 199---207.

\bibitem{Baird1977b} B.~B.~Baird, {\it Epimorphisms of inverse
semigroups of homeomorphisms between closed subsets,} Semigroup
Forum {\bf 14}:2 (1977), 161---166.

\bibitem{Baird1979} B.~B.~Baird, {\it Inverse semigroups of
homeomorphisms are Hopfian,} Canad. J. Math. {\bf 31}:4 (1979),
800---807.

\bibitem{BanakhDimitrova2009xx} T.~Banakh and S.~Dimitrova,
\emph{Openly factorizable spaces and compact extensions of
topological semigroup}, Preprint (http://arxiv.org/abs/0811.4272).

\bibitem{Beida1980} A.~A.~Be{\u{\i}}da, {\it Continuous inverse
semigroups of open partial homeomorphisms,} Izv. Vyssh. Uchebn.
Zaved. Mat. no. \textbf{1} (1980), 64---65 (in Russian).

\bibitem{BoothBrown1978} P.~I.~Booth and R.~Brown, {\it Spaces of
partial maps, fibred mapping spaces and the compact-open
topology,} General Topology Appl. {\bf 8} (1978), 181---195.



\bibitem{CHK} J.~H.~Carruth, J.~A.~Hildebrant and  R.~J.~Koch,
\emph{The Theory of Topological Semigroups}, Vol. I, Marcel
Dekker, Inc., New York and Basel, 1983; Vol. II, Marcel Dekker,
Inc., New York and Basel, 1986.


\bibitem{CP} A.~H.~Clifford and  G.~B.~Preston, \emph{The
Algebraic Theory of Semigroups}, Vol. I., Amer. Math. Soc. Surveys
7, Providence, R.I., 1961; Vol. II., Amer. Math. Soc. Surveys 7,
Providence, R.I., 1967.

\bibitem{DiConcilioNaimpally1998} A.~Di~Concilio nad S.~Naimpally,
\emph{Function space topologies on (partial) maps}, Recent
Progress in Function Spaces, D.~Di~Maio and \v{L}.~Hol\'{a}
(eds.), Quaderni di Mathematica, Vol.~3, Arace, 1998, 1---34.

\bibitem{Engelking1989} R.~Engelking, \emph{General Topology},
2nd ed., Heldermann, Berlin, 1989.

\bibitem{Filippov1996} V.~V.~Filippov,
\emph{Basic topological structures of the theory of ordinary
differential equations}, Topology in Nonlinear Analysis, Banach
Center Publ. \textbf{35} (1996), 171---192.

\bibitem{GierzHofmannKeimelLawsonMisloveScott2003} G. Gierz, K.
H.~Hofmann, K.~Keimel, J.~D.~Lawson, M.~W.~Mislove, and
D.~S.~Scott, \emph{Continuous Lattices and Domains}. Cambridge
Univ. Press, Cambridge, 2003.

\bibitem{Gluskin1955} L. M. Gluskin. {\em  Simple semigroups with
zero}. Dokl. AN SSSR {\bf 103}:1 (1955), 5---8 (in Russian).

\bibitem{Gluskin1977} L.~M.~Gluskin, {\em
Semigroups of homeomorphisms}, Dokl. Akad. Nauk UkrSSR. Ser. A
no.~\textbf{12} (1977), 1059---1061 (in Russian).

\bibitem{GluskinScheinSnepermanYaroker1977} L.~M.~Gluskin,
B.~M.~Schein, L.~B.~\v{S}neperman, and I.~S.~Yyaroker, {\em
Addendum to a survey of semigroups of contionuous selfmaps},
Semigroup Forum \textbf{14} (1977), 95---125.

\bibitem{GutikLawsonRepovs2009} O.~Gutik, J.~Lawson, and
D.~Repov\v{s}, {\em Semigroup closures of finite rank symmetric
inverse semigroups}, Semigroup Forum {\bf 78}:2 (2009), 326---336.

\bibitem{GutikPavlyk2001} O.~V.~Gutik and K.~P.~Pavlyk, {\em
H-closed topological semigroups and Brandt $\lambda-$extensions},
Mat. Metody Fis.-Mekh. Polya {\bf 44}:3 (2001), 20---28 (in
Ukrainian).

\bibitem{GutikPavlyk2003} O.~V.~Gutik and K.~P. Pavlyk, {\em
Topological Brandt $\lambda$-extensions of absolutely $H$-closed
topological inverse semigroups}, Visnyk Lviv Univ., Ser.
Mekh.-Math. \textbf{61} (2003), 98---105.

\bibitem{GutikPavlyk2005} O.~V.~Gutik and K.~P.~Pavlyk, {\em On
topological semigroups of matrix units}, Semigroup Forum {\bf
71}:3 (2005), 389---400.

\bibitem{GutikPavlyk2006} {O. V. Gutik and K.~P. Pavlyk,} {\em
On Brandt  $\lambda^0$-extensions of semigroups with zero}, Mat.
Metody Fis.-Mekh. Polya {\bf 49}:3 (2006), 26---40.

\bibitem{GutikReiter2009} {O. V. Gutik and A.~R. Reiter,}
\emph{Symmetric inverse topological semigroups of finite rank
$\leqslant n$}, Mat. Metody Fis.-Mekh. Polya {\bf 53}:3 (2009) (to
appear) (http://arxiv.org/abs/0912.0198).

\bibitem{GutikRepovs} O.~Gutik and D.~Repovs, \emph{On Brandt
$\lambda^0$-extensions of monoids with zero}, Semigroup Forum (to
appear) (http://arxiv.org/abs/0910.0535v1).

\bibitem{HofmannMisloveStralka}
K.~H.~Hofmann, M.~Mislove and A.~Stralka, \emph{The Pontryagin
Duality of Compact 0-dimensional Semilattices and its
Applications}, Lecture Notes in Math. Vol.~\textbf{396}, Springer,
1974.

\bibitem{Hola1998} \v{L}.~Hol\'{a}, \emph{Topologies on the space
of partial maps}, Recent Progress in Function Spaces, D.~Di~Maio
and \v{L}.~Hol\'{a} (eds.), Quaderni di Mathematica, Vol.~3,
Arace, 1998, 157---186.

\bibitem{Hola1999} \v{L}.~Hol\'{a}, \emph{Complete metrizability of
generalized compact-open topology}, Topology Appl. \textbf{91}:2
(1999), 159---167.

\bibitem{KumziShapiro1997} H.~P.~K\"{u}nzi and L.~B.~Shapiro,
\emph{On simulataneous extension of contionuous partial fuctions},
Proc. Amer. Math. Soc. \textbf{125} (1997), 1853---1859.

\bibitem{Kuratowski1955} K.~Kuratowski, {\em
Sur l'espace des fonctions partielles}, Ann. Mat. Pura Appl.
\textbf{40} (1955), 61---67.

\bibitem{Magill1975} K.~D.~Magill, Jr., {\em
A survey of semigroups of contionuous selfmaps}, Semigroup Forum
\textbf{11} (1975/1976), 189---282.

\bibitem{Mendes-GoncalvesSullivan2006} S.~Mendes-Gon\c{c}alves and
R.~P.~Sullivan, {\em Maximal inverse subsemigroups of the
symmetric inverse semigroup on a finite-dimensional vector
space.}, Comm. Algebra \textbf{34}:3 (2006),  1055---1069.

\bibitem{Orlov1974} S.~D.~Orlov, {\em Topologization of the
generalized group of open partial homeomorphisms of a locally
compact Hausdorff space}, Izv. Vyssh. Uchebn. Zaved. Mat. no.
\textbf{11}(150) (1974), 61---68 (in Russian); English version in:
Russian Mathematics (Izvestiya VUZ. Matematika), \textbf{18}:11
(1974), 47--–52.

\bibitem{Orlov1974a} S.~D.~Orlov, {\em
On the theory of generalized topological groups}, Theory of
Semigroups and its Applications, Sratov Univ. Press. no.
\textbf{3} (1974), 80---85 (in Russian).

\bibitem{Schein1966} { B. M. Schein,} {\em Homomorphisms and
subdirect decompositions of semigroups}, Pacific J. Math. {\bf
24}:3 (1966), 529---547.

\bibitem{Sneperman1962} L.~B.~\v{S}neperman, {\em
Semigroups of contionuous transformations and homeomorphisms of a
simple arc}, Dokl. Akad. Nauk SSSR \textbf{146} (1962),
1301---1304 (in Russian).

\bibitem{Stepp1969} J.~W.~Stepp, {\it A note on maximal locally
compact semigroups,} Proc. Amer. Math. Soc. {\bf 20}:1 (1969),
251---253.

\bibitem{Stepp1975} J.~W.~Stepp, {\it Algebraic maximal
semilattices,} Pacific J. Math. {\bf 58}:1 (1975), 243---248.

\bibitem{Subbiah1987} S.~Subbiah, {\it The compact-open topology
for semigroups of continuous self-maps,} Semigroup Forum {\bf
35}:1 (1987), 29---33.

\bibitem{Wagner1952} V.~V.~Wagner, \emph{Generalized groups},
Dokl. Akad. Nauk SSSR \textbf{84} (1952), 1119---1122 (in
Russian).

\bibitem{Yaroker1972} I.~S.~Yaroker, {\em
Semigroups of homeomorphisms of certain topological spaces}, Dokl.
Akad. Nauk UkrSSR. Ser. A. no.~\textbf{11} (1972), 1008---1010 (in
Russian).
\end{thebibliography}
\end{document}